\documentclass{amsart}
\usepackage[margin=1in]{geometry}
\usepackage[T1]{fontenc}
\usepackage{lmodern,amsmath,amssymb,microtype,indentfirst,needspace}
\usepackage{xcolor}
\usepackage[hidelinks]{hyperref}
\hypersetup{pdftitle={A new criterion for integral modular categorification},pdfauthor={Jingcheng Dong and Sébastien Palcoux}}
\newtheorem{theorem}{Theorem}[section]
\newtheorem{proposition}[theorem]{Proposition}
\newtheorem{corollary}[theorem]{Corollary}
\newtheorem{lemma}[theorem]{Lemma}
\numberwithin{equation}{section}
\DeclareMathOperator*{\lcm}{lcm}
\newcommand{\FPdim}{\mathrm{FPdim}}
\newcommand{\VVec}{\mathrm{Vec}}
\newcommand{\Rep}{\mathrm{Rep}}
\newcommand{\githublink}[2]{\href{#1}{\textcolor{blue}{#2}}}
\newcommand{\suppfile}[2]{\githublink{https://github.com/sebastienpalcoux/Fusion-Categories/blob/14ecf6e059382ef126aa31c49fde9537005aa6e0/NewCritPaper/python/#1}{#2}}
\newcommand{\repodir}[1]{\githublink{https://github.com/sebastienpalcoux/Fusion-Categories/tree/14ecf6e059382ef126aa31c49fde9537005aa6e0/NewCritPaper/python}{#1}}
\allowdisplaybreaks
\title{A new criterion for integral modular categorification}
\author{Jingcheng Dong}
\address{J.~Dong, College of Mathematics and Statistics, Nanjing University of Information Science and Technology, Nanjing 210044, China}
\email{jcdong@nuist.edu.cn}
\author{S\'ebastien Palcoux}
\address{S.~Palcoux, Beijing Institute of Mathematical Sciences and Applications, Huairou District, Beijing, China}
\email{sebastien.palcoux@gmail.com}
\urladdr{\url{https://sebastienpalcoux.github.io/}}
\date{}
\begin{document}

\begin{abstract}
We generalize an argument of Etingof, Nikshych and Ostrik to give a necessary criterion for integral modular categorification. For each nonzero entry of the $S$-matrix, consider the least common multiple of the Frobenius--Perron dimensions of the two simple objects indexing it. In any column, the sum of the squares of these integers is bounded by the dimension of the category. We also obtain a divisibility condition when the column is indexed by a non-invertible simple object of the adjoint subcategory. We apply the criterion to complete the classification of categorifiable integral modular data up to rank~$14$, and up to rank~$25$ in odd dimension.

\end{abstract}
\maketitle
\setcounter{tocdepth}{1}
\tableofcontents

\section{Introduction}
Theorem~\ref{thm:criterion} restricts the dimensions of the simple objects that index the nonzero entries in a column of the $S$-matrix of an integral modular fusion category. It generalizes an argument of Etingof, Nikshych and Ostrik in the proofs of \cite[Lemmas~9.3 and~9.13]{ENO11}. In Lemma~9.3, for instance, column orthogonality forces a simple object of dimension $qr$ to have a nonzero $S$-entry against one of dimension $pq$. The entry is then divisible by $pqr$ in the ring of algebraic integers. For any two simple objects, the same argument gives divisibility by the least common multiple of their dimensions. Averaging the column norm over the Galois group allows us to use these divisibilities simultaneously.

Throughout, $\mathcal C$ is defined over $\mathbb C$ and an integral category is equipped with its canonical positive spherical structure \cite[Corollary~9.6.6]{EGNO15}. We write $\mathcal O(\mathcal C)$ for the set of isomorphism classes of simple objects, and $s_{X,Y}$ for the trace of the double braiding of $X$ and $Y$, so $s_{\mathbf1,Y}=\FPdim(Y)$. The pointed subcategory $\mathcal C_{pt}$ is generated by the invertible simple objects, and the adjoint subcategory $\mathcal C_{ad}$ by the objects $Y\otimes Y^*$ with $Y$ simple. We call $\mathcal C$ \emph{perfect} if the unit is its only invertible object up to isomorphism.

\begin{theorem}\label{thm:criterion}
Let $X$ be a simple object of an integral modular fusion category $\mathcal C$. Let $Y$ range over the non-invertible simple objects, up to isomorphism, for which $s_{X,Y}\ne0$. Then
\begin{equation}\label{eq:support}
 \FPdim(X)^2\FPdim(\mathcal C_{pt})+
 \sum_Y\lcm(\FPdim(X),\FPdim(Y))^2\leq\FPdim(\mathcal C).
\end{equation}
If $X\in\mathcal C_{ad}$ is non-invertible, at least one such $Y$ exists and
\begin{equation}\label{eq:support-gcd}
 \gcd_Y\!\left(\frac{\FPdim(Y)^2}{\gcd(\FPdim(X),\FPdim(Y))}\right)
 \ \bigm|\ \FPdim(\mathcal C_{pt}).
\end{equation}
In particular, if $\mathcal C$ is perfect and $X$ is a non-unit simple object, then
\begin{equation}\label{eq:perfect}
 \gcd_Y\FPdim(Y)=1.
\end{equation}
\end{theorem}

For a non-invertible simple object $X$ of $\mathcal C_{ad}$, \eqref{eq:support-gcd} implies that each prime coprime to $\FPdim(\mathcal C_{pt})$ is coprime to at least one of the dimensions in the sum. The bound \eqref{eq:support} counts each such simple object once, even if its dimension is coprime to several of these primes.

Our application concerns the classification of integral modular fusion categories in low ranks. A \emph{type} is the list of Frobenius--Perron dimensions of the simple objects, counted with multiplicities. Bruillard and Rowell used Egyptian fractions to study the possible types at a given rank~\cite{BrRo}. Ng, Rowell and Wen studied modular data through rank~$12$ without assuming integrality~\cite{NRW}. Alekseyev, Bruns, Palcoux and Petrov classified the categorifiable integral modular data through rank~$13$ and gave exhaustive lists of candidate types at ranks~$14$ and~$15$~\cite{ABPP}. At rank~$14$ their list contains $35$ non-pointed types, one of which is realized. Theorem~\ref{thm:criterion} eliminates $32$ of the \suppfile{data/rank14.txt}{$34$ remaining types}. The \suppfile{README.md\#symmetric-supports}{last two} require the centralizer argument of \S\ref{sec:rank14}.

\begin{corollary}\label{cor:r14}
Every non-pointed integral modular fusion category of rank~$14$ is braided equivalent to $\mathcal{Z}(\VVec(A_4,\omega))$ for some $3$-cocycle $\omega$.
\end{corollary}

Pointed modular categories are classified by finite metric groups up to isometry \cite[Proposition~2.41 and \S2.11.5]{DrGNO}. We count modular data up to simultaneous relabeling, using the twists of the positive spherical structure. Here $\VVec(G,\omega)$ denotes the category of $G$-graded vector spaces with associator given by a $3$-cocycle $\omega$. In odd dimension, the results of Bruillard--Rowell and Czenky--Plavnik give pointedness through rank~$15$ \cite{BrRo,CzPl}. Below rank~$25$, the only non-pointed modular data in odd dimension are three of rank~$17$ and dimension~$225$ \cite[Theorem~1.9 and Remark~12.6]{ABPP}. Their categorifications are described in \cite{GPR24}. At rank~$25$ Czenky, Gvozdjak and Plavnik treated the non-perfect case \cite{CzPl2}, and \cite[Proposition~12.8]{ABPP} left three perfect types to consider.

\begin{corollary}\label{cor:Oddr25}
Every non-pointed odd-dimensional modular fusion category of rank~$25$ is braided equivalent to $\mathcal{Z}(\VVec(C_7\rtimes C_3,\omega))$ for some $3$-cocycle $\omega$, where the semidirect product is nonabelian.
\end{corollary}

In particular, the low-rank classifications give rank at least~$15$ for a non-pointed simple integral modular fusion category, should one exist, and at least~$27$ in odd dimension (the rank is odd, since all non-unit simple objects occur in dual pairs). For related existence questions see \cite[Question~1.8]{ABPP}, \cite[Question~1.3]{LPRinter} and, on weak group-theoreticity, \cite[Question~2]{ENO11}.

The \emph{Anyon Atlas}, which appeared after submission, gives a numerical catalogue through rank~$22$ \cite{Atlas}. A subsequent search extends it through rank~$42$ \cite{Atlas2}. Neither search certifies exhaustiveness. The two corollaries above use the exhaustive low-rank reductions of \cite{ABPP,CzPl2}.

We also apply the criterion to the \suppfile{data/rank15.txt}{rank-$15$ census} in \S\ref{sec:rank15}.

\section{Number-theoretic preliminaries}\label{sec:gal}

An algebraic integer is called \emph{totally positive} if all of its Galois conjugates are positive. A \emph{cyclotomic integer} is an algebraic integer that lies in a cyclotomic field $\mathbb Q(\zeta_n)$ for some $n$, where $\zeta_n=e^{2\pi i/n}$. Let $y$ be a nonzero cyclotomic integer. Since the Galois group of $\mathbb Q(\zeta_n)$ is abelian, complex conjugation commutes with every automorphism $\sigma$ of this field. Hence $\sigma(y\overline y)=\sigma(y)\,\overline{\sigma(y)}$ is positive. Thus $|y|^2$ is a totally positive algebraic integer.

Let $x_1,\dots,x_d$ be the distinct conjugates of a totally positive algebraic integer $x$. Its norm $N(x)=\prod_{i=1}^d x_i$ is a positive rational integer. By the arithmetic--geometric mean inequality,
\[
  \frac{x_1+\cdots+x_d}{d}\geq N(x)^{1/d}\geq1.
\]
For $x=|y|^2$ we may equally take the mean over the whole Galois group of a cyclotomic field containing $y$, since every conjugate of $|y|^2$ occurs the same number of times.

\begin{lemma}\label{lem:lcm}
Let $x$ be an algebraic integer and let $m,n$ be positive integers. If both $x/m$ and $x/n$ are algebraic integers, then so is $x/\lcm(m,n)$.
\end{lemma}
\begin{proof}
Let $g=\gcd(m,n)$. By B\'ezout's identity, there exist integers $u,v$ such that $um+vn=g$. Then $u(x/n)+v(x/m)=gx/(mn)=x/\lcm(m,n)$ is an algebraic integer.
\end{proof}

\section{Proof of Theorem~\ref{thm:criterion}}\label{sec:proof}

\begin{proof}
Fix a simple object $X$. Since the $S$-matrix is symmetric, both $s_{X,Y}/\FPdim(X)$ and $s_{X,Y}/\FPdim(Y)$ are algebraic integers \cite[Proposition~8.13.11(ii)]{EGNO15}. By Lemma~\ref{lem:lcm}, so is
\[
 \frac{s_{X,Y}}{\lcm(\FPdim(X),\FPdim(Y))}.
\]
It belongs to a cyclotomic field containing all the entries of $S$ \cite[Theorem~8.14.7]{EGNO15}. Whenever $s_{X,Y}\ne0$, the squared absolute value of this quotient has Galois mean at least one by \S\ref{sec:gal}.

To prove \eqref{eq:support}, we average the column norm identity \cite[Theorem~8.17.4(iii)]{EGNO15},
\begin{equation}\label{eq:norm2}
 \sum_{Y\in\mathcal O(\mathcal C)}
 \left|\frac{s_{X,Y}}{\FPdim(X)}\right|^2
 =\frac{\FPdim(\mathcal C)}{\FPdim(X)^2}.
\end{equation}
For invertible $Y$, the character indexed by $X$ takes a root of unity as its value $s_{X,Y}/\FPdim(X)$, since $Y$ has finite tensor order. Thus each invertible object contributes one. For a non-invertible $Y$ with $s_{X,Y}\ne0$, the first paragraph bounds the Galois mean of its term from below by $\lcm(\FPdim(X),\FPdim(Y))^2/\FPdim(X)^2$. The right-hand side of \eqref{eq:norm2} is rational, so the asserted inequality follows on multiplying by $\FPdim(X)^2$.

Suppose now that $X\in\mathcal C_{ad}$ is non-invertible. Its column is orthogonal to the unit column, whose entries are $\FPdim(Y)$ \cite[Proposition~8.14.2]{EGNO15}. Moreover $s_{X,Y}=\FPdim(X)$ for every invertible $Y$, since $\mathcal C_{ad}=\mathcal C_{pt}'$ \cite[Corollary~8.22.8]{EGNO15}. Divide the orthogonality relation by $\FPdim(X)$. Separating the invertible terms and collecting the integer coefficients on the left gives
\begin{equation}\label{eq:weighted-orthogonality}
 \sum_Y \frac{\FPdim(Y)^2}{\gcd(\FPdim(X),\FPdim(Y))}
 \frac{s_{X,Y}}{\lcm(\FPdim(X),\FPdim(Y))}
 =-\FPdim(\mathcal C_{pt}),
\end{equation}
where $Y$ ranges over the non-invertible simple objects with $s_{X,Y}\ne0$. The right-hand side is nonzero, so the sum has at least one term. Dividing by the gcd of the integer coefficients on the left gives a rational algebraic integer on the right, hence an integer. This proves \eqref{eq:support-gcd}.

If $\mathcal C$ is perfect, $\mathcal C_{ad}=\mathcal C$ and $\FPdim(\mathcal C_{pt})=1$. The coefficient of the term for $Y$ in \eqref{eq:weighted-orthogonality} is divisible by $\FPdim(Y)$ and divides $\FPdim(Y)^2$. It has the same prime divisors as $\FPdim(Y)$, so \eqref{eq:support-gcd} gives \eqref{eq:perfect}.
\end{proof}

\noindent\textit{Testing a type.}
The type alone does not tell us which simple objects belong to $\mathcal C_{ad}$. The universal grading partitions the type into $\FPdim(\mathcal C_{pt})$ parts with equal sums of squares \cite[Theorem~3.5.2 and \S8.22]{EGNO15}. Every component containing an invertible object has the type of $\mathcal C_{ad}$, since tensoring by that object carries the neutral component onto it \cite[Proposition~3.17]{CzPl2}.

For each possible partition, the \suppfile{criterion.py}{program} tests the columns indexed by non-unit simple objects of $\mathcal C_{ad}$. The witnesses for \eqref{eq:support-gcd} must fit in the bound \eqref{eq:support}. A nonzero entry contributes to the bounds for both its row and its column, so the columns must admit a \suppfile{README.md\#symmetric-supports}{common symmetric support}. The type is excluded if every partition fails.

To use the multiplicities in a type, recall that Galois conjugation permutes the characters $X\mapsto s_{X,Y}/\FPdim(Y)$. It preserves $\FPdim(Y)$ because their formal codegrees $\FPdim(\mathcal C)/\FPdim(Y)^2$ are rational \cite[proof of Theorem~8.14.7]{EGNO15}. Thus for fixed $X$ the quotients $s_{X,Y}/\lcm(\FPdim(X),\FPdim(Y))$ for $Y$ of one dimension have integral sum and integral sum of squared absolute values. The \suppfile{README.md\#equal-dimension-bounds}{repository guide} gives the resulting bounds. At rank $15$ we use them when $15\FPdim(X)^2\ge\FPdim(\mathcal C)$.

Consider a \suppfile{data/rank15.txt\#L5282}{rank-$15$ type} of global dimension $4\cdot12600^2$. There are three simple objects of dimension $12600$ and one of dimension $2520$. Each of these four columns must meet a row indexed by a non-unit simple object of dimension prime to $3$ by \eqref{eq:perfect}. Only $175$ and $2800$ are prime to $3$, and $2800$ has least common multiple $25200$ with either column dimension. This leaves the unique simple object of dimension $175$. Its column already contains four terms of $12600^2$ by symmetry, with the unit still to count.

\Needspace{8\baselineskip}
\section{Applications}\label{sec:appl}
\subsection{Proof of Corollary~\ref{cor:r14}}\label{sec:rank14}

\begin{proposition}
A non-pointed non-perfect integral modular fusion category of rank $14$
is braided equivalent to $\mathcal{Z}(\VVec(A_4,\omega))$ for some $3$-cocycle $\omega$.
\end{proposition}
\begin{proof}
By \cite[Proposition~11.1]{ABPP}, there are eight possible types.
Seven are listed in the \suppfile{data/rank14.txt\#L1-L7}{first seven rows of \texttt{rank14.txt}}, in the order of that proposition.
Of these seven, Theorem~\ref{thm:criterion} excludes all but the first two. The third type is ruled out as follows by hand.

\begin{lemma}\label{lem:first}
There is no modular fusion category of type
$t=[1,1,24,24,36,40,45,45,90,90,90,180,180,180]$.
\end{lemma}
\begin{proof}
Assume that $\mathcal C$ has type $t$. The adjoint subcategory $\mathcal C_{ad}$ is the neutral component of the universal grading, which here has two components of equal dimension \cite[Proposition~2.3 and Lemma~3.31]{DrGNO}, \cite[Theorem~3.5.2]{EGNO15}. Thus $\FPdim(\mathcal C_{ad})=\FPdim(\mathcal C)/2=2\cdot180^2$. The sum of the squares of all entries of $t$ smaller than $180$ equals exactly $180^2$. Hence any sublist containing $1$ whose squares sum to half the total must consist of all these smaller entries together with exactly one $180$.

Let $X\in\mathcal O(\mathcal C)$ be the simple object with $\FPdim(X)=36$. By the preceding argument, $X\in\mathcal C_{ad}$. Since $3$ is coprime to $\FPdim(\mathcal C_{pt})=2$, \eqref{eq:support-gcd} gives a non-invertible simple object $Y$ with $s_{X,Y}\ne0$ and $3\nmid\FPdim(Y)$. The only possible dimension for $Y$ is $40$. However,
\[
 \lcm(36,40)^2+36^2\times2=132192>129600=\FPdim(\mathcal C),
\]
which contradicts \eqref{eq:support}.
\end{proof}

The other four exclusions use the same grading argument.
In each of \suppfile{data/rank14.txt\#L4-L7}{rows $4$--$7$} the largest dimension occurs three times,
and its square equals the sum of the smaller squares. Choose a simple object $X$ of $\mathcal C_{ad}$
of this largest dimension and take $p=3,5,3,3$, respectively.
For every non-invertible simple object $Y$ of dimension coprime to $p$,
the square of $\lcm(\FPdim(X),\FPdim(Y))$ already equals $\FPdim(\mathcal C)$.

Assume now that $\mathcal C$ is of either of the first two types. The five simple objects of dimensions $1,1,2,3,3$ span a fusion subcategory $\mathcal D$. Indeed, their pairwise tensor products have dimension at most $9$, whereas every other simple object has dimension at least $24$.

\begin{lemma}
The fusion subcategory $\mathcal D$ is Tannakian and equivalent to $\Rep(S_4)$.
\end{lemma}
\begin{proof}
By \cite[Theorem~3.2]{Mug03} or \cite[Theorem~3.14]{DrGNO},
\begin{equation}\label{eq:DD'}
 \FPdim(\mathcal D)\,\FPdim(\mathcal D')=\FPdim(\mathcal C).
\end{equation}
The dimension of $\mathcal D'$ is therefore $1176$ or $15000$. In both cases the first seven squares sum to this number, and the eighth square is already too large. Thus $\mathcal D'$ has type $[1,1,2,3,3,24,24]$ or $[1,1,2,3,3,24,120]$, respectively. In particular $\mathcal D\subset\mathcal D'$ and $\mathcal D$ is symmetric. By Deligne's theorem \cite[Theorem~2.43]{DrGNO}, the underlying group has order $24$ and five irreducible characters. The only such group is $S_4$ \cite[proof of Proposition~III.1]{BO18}. Its center is trivial, so $\mathcal D$ is Tannakian.
\end{proof}

The M\"uger center of $\mathcal D'$ is $\mathcal D'\cap\mathcal D''=\mathcal D$. De-equivariantizing $\mathcal D'$ by $\Rep(S_4)$ gives a modular fusion category of dimension $\FPdim(\mathcal C)/24^2$, hence of dimension $49$ or $625$ \cite[Propositions~4.26 and~4.56]{DrGNO}. It is integral by oddness \cite[Corollary~9.6.10]{EGNO15}. It is pointed in both cases. For dimension $7^2$ this follows because the square of the dimension of every simple object divides $7^2$ \cite[Proposition~8.14.6]{EGNO15}, and for dimension $5^4$ it is \cite[Lemma~4.11]{DN18}.

Let $H$ be the abelian group of simple objects of the pointed category. Each $S_4$-orbit on $H$ gives at least one simple object in $\mathcal D'$, and the orbit of the unit gives the five irreducible representations of $S_4$ \cite[Corollary~2.13 and Remark~3.2]{BN13}. Since $\mathcal D'$ has rank $7$, there are at most two further orbits. This forces $|H|\leq1+2\cdot24=49$. In the case of equality every nonidentity element has trivial stabilizer. A transposition fixes the product of any element with its image, so it acts by inversion. The product of two distinct transpositions would then fix all of $H$, a contradiction.

The eighth type has three invertible objects, five simple objects of dimension $3$,
and six of dimension $4$. By \cite[Proposition~11.1]{ABPP}, its six
possible modular data all occur in the centers
$\mathcal Z(\VVec(A_4,\omega))$. In each such center the four simple objects of
the canonical $\Rep(A_4)$ are closed under taking duals and tensor
constituents, and each has trivial twist. The Verlinde formula and the
twists give the same properties on the corresponding labels in $\mathcal C$. By the balancing equation they span a symmetric
subcategory of type $[1,1,1,3]$. The group in Deligne's theorem has
order $12$ and abelianization of order $3$. Its Sylow $3$-subgroup
cannot be normal, since the quotient of order four would be abelian.
There are thus four such subgroups, each its own normalizer. Conjugation
on them embeds the group in $S_4$ with image $A_4$. Since $A_4$ is
centerless, the subcategory is Tannakian. Its dimension is
$12=\sqrt{\FPdim(\mathcal C)}$, so it is Lagrangian. The proposition
follows from \cite[Remark~4.65 and Theorem~4.66]{DrGNO}.
\end{proof}

\begin{proposition}
Every perfect integral modular fusion category of rank at most $14$ is trivial.
\end{proposition}
\begin{proof}
The assertion up to rank $13$ is \cite[Theorem~1.5]{ABPP}. At rank $14$
there remain $27$ perfect types by \cite[Sec.~XI]{ABPP}, listed in
\suppfile{data/rank14.txt\#L8-L34}{rows $8$--$34$ of \texttt{rank14.txt}}.
Here $\mathcal C_{ad}=\mathcal C$. For each row the
\suppfile{verify.py}{Python verification} finds the dimension of a simple object $X$
and a prime $p\mid\FPdim(X)$ such that
$\lcm(\FPdim(X),\FPdim(Y))^2=\FPdim(\mathcal C)$ for every non-unit
simple object $Y$ of dimension coprime to $p$. By \eqref{eq:perfect} some such
$Y$ occurs in the column of $X$. Its contribution together with the unit
contradicts \eqref{eq:support} in each of the $27$ cases.
\end{proof}

\subsection{Proof of Corollary~\ref{cor:Oddr25}}\label{sec:rank25}

As observed in \cite[Sec.~XII]{ABPP}, the classification in odd dimension at rank
$25$ reduces to the perfect case.

\begin{proposition}\label{prop:R25}
There is no perfect odd-dimensional modular fusion category of rank $25$.
\end{proposition}

\begin{proof}
By \cite[Proposition~12.8]{ABPP}, there remain \suppfile{data/rank25-odd.txt}{three possible types}. Suppose that the first type is realized. Choose a simple object $X$ with $\FPdim(X)=4095=3^2\cdot5\cdot7\cdot13$. Take $p=13$. Since $\mathcal C$ is perfect, the dimensions of the non-unit simple objects in the support of this column have gcd one by \eqref{eq:perfect}. Thus $s_{X,Y}\ne0$ for some non-invertible $Y$ whose dimension is coprime to $p$. The only possibilities are $75$ and $175$. However,
\[
 \lcm(4095,75)^2=\lcm(4095,175)^2=20475^2=\FPdim(\mathcal C).
\]
This contradicts \eqref{eq:support}, since the unit contributes a further $4095^2$ to the bound. For the second type take $\FPdim(X)=2925$ and $p=5$. Only $91$ is possible for $\FPdim(Y)$, and $\lcm(2925,91)=20475$. For the third take $\FPdim(X)=135$ and $p=3$. This time only $385$ is possible, and $\lcm(135,385)^2=10395^2=\FPdim(\mathcal C)$. In both cases the same contradiction follows.
\end{proof}

The same test works in \suppfile{verify.py}{$88$ of the $91$ rows} of the \suppfile{data/rank25-odd-91.txt}{original list}. For rows $8$, $26$ and $37$, take $p=3,3,11$, respectively. Apart from the unit, just two simple objects have dimension prime to $p$. They form a dual pair, of dimension $115$, $65$ and $51$ in the three cases. By \eqref{eq:perfect} every non-unit column meets this pair, and it meets both members because $s_{X,Y^*}=\overline{s_{X,Y}}$ \cite[Theorem~8.17.4(i),(iii)]{EGNO15}. By symmetry either member has a column with no zero. In \eqref{eq:support} every row would then contribute at least its squared dimension and the unit would contribute strictly more, a contradiction.

\begin{proof}[Proof of Corollary~\ref{cor:Oddr25}]
By Proposition~\ref{prop:R25} it remains to consider a non-pointed,
non-perfect category. Czenky, Gvozdjak and Plavnik reduce its type to
three invertible objects, sixteen simple objects of dimension $3$ and six of
dimension $7$ \cite[Proposition~5.2 and Lemma~5.4]{CzPl2}.
Two details of that reduction merit mention. In their Lemma~4.2, if
$(\mathcal C_{ad})_{pt}=\VVec$, then $\mathcal C_{ad}$ is perfect and
its M\"uger center $\mathcal C_{ad}\cap\mathcal C_{pt}$ is trivial.
Thus, if $\mathcal C_{ad}$ is nontrivial, the rank bound of $25$ used
there follows from \cite[Theorem~12.2]{ABPP}. In their Lemma~5.3,
the multiple of six is the number of simple objects of dimension three moved
by the group of invertible objects. Each orbit of size three is paired with
its dual, since an orbit of odd size stable under duality would contain
a self-dual simple object. In odd dimension the unit is the only self-dual simple object
\cite[Proposition~4.3]{BrRo}.

The dimension is therefore $3+16\cdot3^2+6\cdot7^2=441$. Since
$3\nmid(7+1)$, \cite[Theorem~4.6 and Proposition~5.1]{GPR24} identifies
$\mathcal C$ with a twisted center for the unique nonabelian group of
order $21$, namely $C_7\rtimes C_3$. Every twist has rank $25$.
Indeed, the identity class contributes the five ordinary irreducible
characters, and the four other classes have cyclic centralizers of
orders $7,7,3,3$. Projective representations of a cyclic group over
$\mathbb C$ are linearizable, so these classes contribute another
$7+7+3+3=20$ simple objects. The three cohomology classes give distinct twist matrices up to relabeling \cite[proof of Lemma~3.2 and Proposition~4.3]{MS21}.
\end{proof}

\subsection{Rank 15}\label{sec:rank15}

On the \suppfile{data/rank15.txt}{rank-$15$ census} of \cite[Sec.~XI]{ABPP}, the tests of \S\ref{sec:proof} exclude $354$ of the $399$ non-perfect types and $8378$ of the $8628$ perfect types. The \suppfile{data/rank15-manuscript-survivors.txt}{$295$ types passing these tests} comprise $45$ non-perfect and $250$ perfect types.

\section{GitHub repository}\label{sec:repository}

The \suppfile{README.md}{repository guide} links the input lists and describes the tests used in \S\ref{sec:appl}. From the \repodir{repository's Python directory}, run \suppfile{reproduce.py}{\texttt{python3 reproduce.py -{}-manuscript}} to reproduce the counts and \suppfile{verify.py}{\texttt{python3 verify.py}} to verify the lists independently. Both scripts use the Python standard library.

\section*{Acknowledgments}
We thank the referees for their careful reading and suggestions, in particular the suggestion to average the column norm. S.~Palcoux is supported by NSFC grant 12471031. GPT-6.2 Astra assisted with the revision and with implementing and checking the computations.


\begin{thebibliography}{99}
\bibitem{ABPP}
M.~A. Alekseyev, W.~Bruns, S.~Palcoux and F.~V. Petrov,
Classification of integral modular data up to rank 13,
\emph{J. Math. Phys.} \textbf{67} (2026), 081701.

\bibitem{Atlas}
J.~Berean-Dutcher, C.~C\'ordova, J.~Parra-Martinez and R.~Tuddenham,
The Anyon Atlas, \href{https://arxiv.org/abs/2610.25439}{arXiv:2610.25439} (2026).

\bibitem{Atlas2}
J.~Berean-Dutcher, C.~C\'ordova, J.~Parra-Martinez, R.~Tuddenham and Z.~Zhu,
Learning the Anyon Atlas, \href{https://arxiv.org/abs/2611.24752}{arXiv:2611.24752} (2026).

\bibitem{BO18}
P.~Bruillard and C.~M. Ortiz-Marrero, Classification of rank 5 premodular categories,
\emph{J. Math. Phys.} \textbf{59} (2018), 011702.

\bibitem{BrRo}
P.~Bruillard and E.~C. Rowell, Modular categories, integrality and Egyptian fractions,
\emph{Proc. Amer. Math. Soc.} \textbf{140} (2012), 1141--1150.

\bibitem{BN13}
S.~Burciu and S.~Natale, Fusion rules of equivariantizations of fusion categories,
\emph{J. Math. Phys.} \textbf{54} (2013), 013511.

\bibitem{CzPl2}
A.~Czenky, W.~Gvozdjak and J.~Plavnik, Classification of low-rank odd-dimensional modular categories,
\emph{J. Algebra} \textbf{655} (2024), 223--293.

\bibitem{CzPl}
A.~Czenky and J.~Plavnik, On odd-dimensional modular tensor categories,
\emph{Algebra Number Theory} \textbf{16} (2022), 1919--1939.

\bibitem{DN18}
J.~Dong and S.~Natale, On the classification of almost square-free modular categories,
\emph{Algebr. Represent. Theory} \textbf{21} (2018), 1353--1368.

\bibitem{DrGNO}
V.~Drinfeld, S.~Gelaki, D.~Nikshych and V.~Ostrik,
On braided fusion categories I, \emph{Selecta Math. (N.S.)} \textbf{16} (2010), 1--119.

\bibitem{EGNO15}
P.~Etingof, S.~Gelaki, D.~Nikshych and V.~Ostrik,
\emph{Tensor Categories}, Math. Surveys Monogr. \textbf{205}, Amer. Math. Soc., Providence, RI, 2015.

\bibitem{ENO11}
P.~Etingof, D.~Nikshych and V.~Ostrik,
Weakly group-theoretical and solvable fusion categories,
\emph{Adv. Math.} \textbf{226} (2011), 176--205.

\bibitem{GPR24}
C.~Galindo, J.~Plavnik and E.~C. Rowell, Integral non-group-theoretical modular categories of dimension $p^2q^2$,
\emph{Bull. Belg. Math. Soc. Simon Stevin} \textbf{31} (2024), 516--525.

\bibitem{LPRinter}
Z.~Liu, S.~Palcoux and Y.~Ren, Interpolated family of non-group-like simple integral fusion rings of Lie type,
\emph{Internat. J. Math.} \textbf{34} (2023), 2350030.

\bibitem{MS21}
M.~Mignard and P.~Schauenburg, Modular categories are not determined by their modular data,
\emph{Lett. Math. Phys.} \textbf{111} (2021), 60.

\bibitem{Mug03}
M.~Müger, On the structure of modular categories,
\emph{Proc. London Math. Soc.} (3) \textbf{87} (2003), 291--308.

\bibitem{NRW}
S.-H.~Ng, E.~C. Rowell and X.-G.~Wen,
Classification of modular data up to rank 12,
\href{https://arxiv.org/abs/2308.09670v2}{arXiv:2308.09670v2} (2025).
\end{thebibliography}
\end{document}